\documentclass[12pt]{amsart}

\usepackage[utf8]{inputenc}
\usepackage[english]{babel}
\usepackage{amsmath}
\usepackage{amsfonts}
\usepackage{amssymb}
\usepackage{graphicx}
\usepackage{mathtools}
\usepackage{amsthm}
\newtheorem{thm}{Theorem}[section]
\newtheorem{prop}[thm]{Proposition}
\newtheorem{lem}[thm]{Lemma}

\newtheorem{rem}[thm]{Remark}

\newtheorem{conj}{Conjecture}
\theoremstyle{definition}
\newtheorem{defi}{Definition}[section]
\newcommand{\R}{\mathbb{R}}
\renewcommand{\P}{\mathbb{P}}
\newcommand{\N}{\mathbb{N}}
\newcommand{\Z}{\mathbb{Z}}
\newcommand{\C}{\mathbb{C}}
\newcommand{\Y}{\mathbb{Y}}

\newcommand{\E}{\mathbb{E}}
\newcommand{\Conf}{\mathrm{Conf}}

\newcommand{\s}{\mathbf{s}}
\newcommand{\x}{\mathbf{x}}
\newcommand{\h}{\mathbf{h}}
\newcommand{\e}{\mathbf{e}}
\newcommand{\p}{\mathbf{p}}

\newcommand{\uuu}{\mathbf{u}}

\newcommand{\vvv}{\mathbf{v}}

\usepackage{fullpage}

\usepackage{hyperref}

\title{A determinantal point process governed by an integrable projection kernel is Giambelli compatible}
\author[A. I. Bufetov]{Alexander I. Bufetov}
\address{CNRS, UMR 7373, \'Ecole centrale de Marseille, Institut de Math\'ematiques de Marseille, Aix-Marseille Universit\'e, Marseille, France \\
Steklov Mathematical Institute of Russian Academy of Sciences, Moscow, Russia \\
Institute for Information Transmission Problems, Russian Academy of Sciences, Moscow, Russia}
\author[P. Lazag]{Pierre Lazag}
\address{CNRS, UMR 7373, \'Ecole centrale de Marseille, Institut de Math\'ematiques de Marseille, Aix-Marseille Universit\'e, Marseille, France}
\date{}
\begin{document}

\begin{abstract}
The first main result of this note, Theorem \ref{thm1}, establishes the determinantal identities  \eqref{form:mainformula} and \eqref{form:mainformulatilde} for the expectation, under a determinantal point process governed by an integrable projection kernel, of scaling limits of  characteristic polynomials sampled at several points.  The determinantal identities \eqref{form:mainformula} and \eqref{form:mainformulatilde} can be seen as the scaling limit of the identity of 
Fyodorov and Strahov for the averages of ratios of products of the values of the characteristic polynomial
of a Gaussian unitary matrix. Borodin, Olshanski and Strahov derived the determinantal identity of Fyodorov and Strahov from the stability of the Giambelli formula  under averaging. In  Theorem \ref{thm2} the stability of the Giambelli formula under averaging is established for  determinantal point process with  integrable projection kernels. The  proof of Theorems \ref{thm1} and \ref{thm2} relies on the characterization of conditional measures of our point processes as orthogonal polynomial ensembles.
\end{abstract}
\maketitle
\section{Introduction}
\subsection{Formulation of the main results}
The main results of this note are Theorem \ref{thm1} and its equivalent reformulation, Theorem \ref{thm2} . In the spirit of Fyodorov and Strahov, Theorem \ref{thm1}, described in Section \ref{sec:charpolintro}, establishes the determinantal identities \eqref{form:mainformula} and \eqref{form:mainformulatilde} for averages of products of ratios of scaling limits of characteristic polynomials for determinantal point processes induced by integrable projection kernels. The determinantal identities \eqref{form:mainformula} and \eqref{form:mainformulatilde} are expressed in terms of Giambelli compatibility introduced by Borodin, Olshanski and Strahov in \cite{giambelli}, see Theorem \ref{thm2}  in Section \ref{sec:giambelliintro}.
\subsubsection{An identity for analogues of characteristic polynomials for determinantal point processes} \label{sec:charpolintro}
Fyodorov and Strahov proved in \cite{fyodorovstrahov1}, \cite{fyodorovstrahov2} that, if $P(z)$ is the characteristic polynomial of a Gaussian Hermitian matrix, then we have
\begin{align} \label{form:charpol}
\E_\P \left[ \frac{P(z_1) \dots P(z_n)}{P(w_1) \dots P(w_n)} \right] = \det\left( \frac{1}{z_i -w_j} \right)_{1 \leq i,j \leq k}^{-1} \det \left( \frac{1}{z_i-w_j}\E_\P\left[ \frac{P(z_i)}{P(w_j)} \right] \right)_{i,j=1}^n.
\end{align}

Our first main result is a sufficient condition for determinantal point processes on a subset of $\R$ to satisfy an analogue of the above identity (\ref{form:charpol}). Let $E$ be a subset of $\R$, either open or countable and without accumulation points. The space of configurations on $E$, denoted by $\Conf(E)$, is the set of all subsets of $E$ without accumulation points. Let $\mu$ be a Radon measure on $E$ and let $\P$ be a determinantal point process on $(E,\mu)$ with correlation kernel $K$, see section \ref{sec:det} below. We assume that the kernel $K$ induces an orthogonal projection on $L^2(E,d\mu)$, and that we have 
\begin{align} \label{cond:hilbertschmidt}
\int_E \frac{K(x,x)}{1+x^2}d\mu(x) <+\infty.
\end{align}
Our first task is to regularize the products 
\begin{align*}
\prod_{x \in X} \frac{z-x}{w-x}
\end{align*}
which might diverge. For $ R >0$ and $u \in \C$ such that $u - i R \notin \R$, we introduce the function 
\begin{align}
g_u : \R &\rightarrow \C \\ \label{eq:defg}
 x &\mapsto 1 - \frac{u}{x + i R}.
\end{align}
Observe that, for $z, w \in \C \setminus \R$, the ratio $g_{z+ i R}/g_{w + i R}$ does not depend on $R$ as we have 
\begin{align*}
\frac{g_{z+ i R}(x)}{g_{w + i R}(x)} = \frac{z-x}{w-x}.
\end{align*}
In what follows, $\log$ stands for the principal branch of the logarithm, defined on $\C \setminus (-\infty,0]$. The following Proposition is proved in Section \ref{sec:proof1}.
\begin{prop} \label{prop:1} Let $\P$ be a determinantal point process on $(E,\mu)$ induced by an orthogonal projection $K$ satisfying \eqref{cond:hilbertschmidt}.
\begin{enumerate}
\item[$(i)$]  For any $u \in \C \setminus \R$, the limits
\begin{align} \label{lim:prop11}
\Psi_X(u)&:=\lim_{ T \rightarrow + \infty} \left\{ \exp \left( - \int_{-
T}^{T} \log g_{u+ i R}(x) K(x,x) dx \right)  \prod_{x \in X, \hspace{0.1cm} |x| \leq T } g_{u+ i R}(x)\right\} \\ \label{lim:prop12}
\tilde{\Psi}_X(u)&:= \lim_{ T \rightarrow + \infty}  \left\{ \exp \left( - u\int_{-
T}^{T} \frac{1}{x+iR} K(x,x) dx \right)  \prod_{x \in X, \hspace{0.1cm} |x| \leq T } g_{u+ i R}(x)\right\}
\end{align}
exist in $L^1(\Conf(E),\P)$.
\item[$(ii)$] For any $n \in \N$, and any $z_1,\dots, z_n,w_1, \dots w_n \in \C\setminus \R$, the algebra generated by the random variables\begin{align*}
\Psi_X(z_j)^{\pm 1}, \hspace{0.1cm} \Psi_X(w_j)^{\pm 1},
\end{align*}
and the algebra generated by
\begin{align*}
\tilde{\Psi}_X(z_j)^{\pm 1}, \hspace{0.1cm} \tilde{\Psi}_X(w_j)^{\pm 1},
\end{align*}
are contained in $L^1(\Conf(E), \P)$.
\item[$(iii)$]There exists a Borel set $\mathcal{W} \subset \Conf(E)$ satisfying $\P(\mathcal{W})=1$ such that for any compact set $\mathcal{K} \subset \C \setminus \R$, there exists an increasing sequence $(T_N)_{N \in \N}$ tending to $+ \infty$ as $N \rightarrow + \infty$ along which the convergences in (\ref{lim:prop11}) and (\ref{lim:prop12}) are uniform for $u \in \mathcal{K}$ for all $X \in \mathcal{W}$. The functions $u \mapsto \Psi_X(u)$  and $u \mapsto \tilde{\Psi}_X(u)$ are analytic for all $X \in \mathcal{W}$.
\end{enumerate}
\end{prop}
We focus on determinantal point processes governed by an orthogonal projection with correlation kernel 
having an integrable form in the following sense.
\begin{defi} \label{def:integrable}We say that the kernel $K$ has an \emph{integrable form} if there exist an open set $U\supset E$ and smooth functions $A,B : U \rightarrow \R$ such that 
\begin{align*}
K(x,y)&= \frac{A(x)B(y)-A(y)B(x)}{x-y}, \quad x \neq y \\
K(x,x) &=A'(x)B(x)-A(x)B'(x).
\end{align*}
\end{defi}
Our first main result is the following Theorem. The symbol $\E_{\P}$ stands for the expectation under $\P$. The definition of rigidity, in the sense of Ghosh and Peres,  is recalled in Section \ref{sec:ppcond}, Definition \ref{def:rigid}.
\begin{thm} \label{thm1} If $\P$ is a rigid determinantal point process on $(E,\mu)$ induced by a projection kernel having an integrable form, then for all $n \in \N$ and all pairwise distinct non-real complex numbers $z_1,\dots z_n,w_1,\dots w_n \in \C \setminus \R$, we have
\begin{align} \label{form:mainformula}
\E_\P \left[ \prod_{i=1}^n \frac{\Psi_X(z_i)}{\Psi_X(w_i)} \right] = \det \left( \frac{1}{z_i-w_j} \right)_{1 \leq i,j \leq n}^{-1} \det \left( \frac{1}{z_i-w_j} \E_\P \left[ \frac{\Psi_X(z_i)}{\Psi_X(w_j)} \right] \right)_{i,j=1}^n.
\end{align}
and
\begin{align} \label{form:mainformulatilde}
\E_\P \left[ \prod_{i=1}^n \frac{\tilde{\Psi}_X(z_i)}{\tilde{\Psi}_X(w_i)} \right] = \det \left( \frac{1}{z_i-w_j} \right)_{1 \leq i,j \leq n}^{-1} \det \left( \frac{1}{z_i-w_j} \E_\P \left[ \frac{\tilde{\Psi}_X(z_i)}{\tilde{\Psi}_X(w_j)} \right] \right)_{i,j=1}^n.
\end{align}
\end{thm}
\begin{rem}In order to consider the case when some of the $z_j$ (or $w_j$) are equal, we may use the following formula, which is a direct application of the Cauchy determinant formula 
\begin{align*}
\det \left( \frac{1}{z_i-w_j} \right)_{1 \leq i,j \leq k}^{-1} = \frac{\prod_{1 \leq i , j \leq k}(z_i-w_j)}{\prod_{ 1 \leq i <j \leq k}(z_i-z_j)(w_j-w_i)},
\end{align*}
and then apply L'H\^opital rule. For example, we have
\begin{align*}
\E_\P \left[ \frac{\Psi_X(z_1)^2}{\Psi_X(w_1)\Psi_X(w_2)} \right] = \frac{(z_1-w_1)^2(z_1-w_2)^2}{w_1-w_2} \frac{\partial}{\partial z_2} \left. \left\{ \det \left( \frac{1}{z_i-w_j} \E_\P \left[ \frac{\Psi_X(z_i)}{\Psi_X(w_j)} \right] \right)_{i,j=1}^2 \right\} \right|_{z_2=z_1}.
\end{align*}
\end{rem}

\subsubsection{Giambelli compatible point processes} \label{sec:giambelliintro}
Theorem \ref{thm1} can be restated in terms of Giambelli compatibility. The formalism of \emph{Giambelli compatible point processes} was introduced by Borodin-Olshanski-Strahov in \cite{giambelli} . Let $\Lambda$ be the algebra over $\C$ of symmetric functions, i.e. of symmetric polynomials of finite degree in infinitely many formal variables $\x_1,\x_2,\dots$, see \cite{macdonald}. We consider homomorphisms of the algebra $\Lambda$  into the algebra of functions on the space of configurations on a Polish space $E$, that is, the space $\Conf(E)$ defined as the space of all subsets of $E$ without accumulation points and endowed with the natural Borel sigma-algebra, see Section \ref{sec:ppcond} below.
\begin{defi}An algebra homomorphism $\rho$ from $\Lambda$ to the algebra of Borel functions on $\Conf(E)$ is called a \emph{specialization}. Given a Borel probability measure $\P$ on $\Conf(E)$, we say that the specialization $\rho$ is an $L^1$-\emph{specialization} if its image is contained in $L^1(\Conf(E), \P)$.
\end{defi}
We extend a specialization $\rho$ to the algebra of power series with coefficients in $\Lambda$, denoted by $\Lambda[[z]]$, by setting
\begin{align*}
\rho(F)(z) := \sum_{k=0}^{+\infty} \rho(F_k)(X)z^k, \quad F= \sum_{k=0}^{+\infty} F_kz^k \in \Lambda [[z]], \hspace{0.1cm} X \in \Conf(E).
\end{align*}
A distinguished basis of the algebra $\Lambda$ is formed by the Schur functions $\s_\lambda$, indexed by Young diagrams $\lambda$, cf. Section \ref{sec:symfunctions}. Using the Frobenius notation, write $\lambda= (p_1, \dots, p_d | q_1, \dots, q_d)$. The Schur function $\s_\lambda$ satisfies the Giambelli formula 
\begin{align*}
\s_\lambda = \det ( \s_{(p_i|q_j)} )_{i,j=1}^d.
\end{align*}
A \emph{Giambelli compatible} point process is then a point process for which the Giambelli formula is stable under averaging. More precisely, we have
\begin{defi}[Borodin-Olshanski-Strahov, \cite{giambelli}] \label{def:giambelli} Let $\P$ be a point process on $E$ and let $\rho$ be an $L^1$-specialization. We say that the pair $(\P, \rho)$ is \emph{Giambelli compatible} if, for any partition $\lambda= (p_1,\dots , p_d |q_1, \dots, q_d)$, we have 
\begin{align} \label{eq:giambelli0}
\E_{\P} [ \rho (\s_\lambda) ] = \det \left( \E_{\P} [ \rho(\s_{p_i,|q_j}) ] \right)_{i,j=1}^d.
\end{align}
\end{defi}

Under certain technical conditions on the specialization $\rho$, the Giambelli compatibility is equivalent to the validity of formula (\ref{form:charpol}), where the characteristic polynomials $P(z)$ are replaced by the specialization of the generating series of the complete homogeneous and the elementary symmetric functions, see Proposition \ref{prop:propgiambelli}. Namely, we introduce the notion of \emph{uniform $L^1$-specialization} for which the specializations of these generating series converge and belong to $L^1$, see Section \ref{sec:giambelli} below.

For $R >0$, consider the specialization $\rho^R$ given by 
\begin{align*}
\p_k \mapsto \rho^R(\p_k) : = \sum_{x \in X} \frac{1}{(x+ i R)^k} - \E_\P \sum_{x \in X} \frac{1}{(x+ i R)^k}, \quad k=1,2,\dots
\end{align*}
and the specialization $\tilde{\rho}^R$ given by
\begin{align*}
\tilde{\rho}(\p_1) = \rho (\p_1), \quad \tilde{\rho}(\p_k) = \sum_{x \in X} \frac{1}{(x+ i R)^k}, \quad k \geq 2,
\end{align*}
where the $\p_k$ are the Newton power sums, see section \ref{sec:symfunctions}. It will be seen that if $\P$ is a determinantal point process for which the correlation kernel is the kernel of an orthogonal projection satisfying (\ref{cond:hilbertschmidt}), then $\rho^R$ and $\tilde{\rho}^R$ are uniform $L^1$-specializations. It will also be seen, see Proposition \ref{prop:analyticcont}, that the image by $\rho^R$ of the generating series of the elementary symmetric functions $E(u)$ (resp. of the complete homogeneous symmetric fuctions $H(u)$) is the random function $\Psi_X(-u)$ (resp. the random function $\Psi_X^{-1}(u)$) defined in Section \ref{sec:charpolintro} above, while the image by $\tilde{\rho}^R$ of the series $E(u)$ (resp. $H(u)$) is the function $\tilde{\Psi}_X(-u)$ (resp. $\tilde{\Psi}_X^{-1}(u)$). Theorem \ref{thm1} is then equivalent to the following 
\begin{thm}\label{thm2} If $\P$ is a rigid determinantal point process on $(E,\mu)$ induced by an orthogonal projection admitting an integrable kernel $K$ satisfying (\ref{cond:hilbertschmidt}), then the pairs $(\P,\rho^R)$ and $(\P,\tilde{\rho}^R)$ are Giambelli compatible.
\end{thm}
Our argument relies on a proposition  from \cite{bufetovconditional} to the effect that determinantal point processes with integrable projection kernels conditioned to be equal to a specific fixed configuration outside a compact interval are orthogonal polynomial ensembles, cf. Proposition \ref{prop:bufetovconditional} below.

\begin{rem} A characterization of integrable projection kernels is provided by \cite{bufetovromanov} and \cite{bufetovquasisymmetries}. A kernel $K$ of orthogonal projection onto a space $L \subset L^2(E,d\mu)$ has an integrable form if and only if the space $L$ satisfies the following division property: if $f \in L$ and $p \in E$ is such that $f(p)=0$, then the function
\[x \mapsto \frac{f(x)}{x-p} \]
belongs to $L$.
\end{rem}
Examples of determinantal point processes with an integrable projection kernel include the sine process, the Airy process, the Bessel process (see \cite{bufetovrigid}), as well as determinantal point processes governed by reproducing kernels of de Branges spaces (see \cite{bufetovshirai}). In the discrete setting, the discrete sine process (\cite{boo}), the process with the Gamma kernel (\cite{bogamma}) and the process with the discrete Bessel kernel (\cite{borodinbessel}) also form examples.

See for example \cite{borodindet} for a survey on determinantal point processes.

Borodin, Olshanski and Strahov \cite{giambelli} (see also Borodin-Strahov \cite{borodinstrahov}) derive the determinantal property from Giambelli compatibility.
\begin{conj}Let $\P$ be a point process on $E$ such that $\rho^R$ is an $L^1$-specialization. The pair $(\P,\rho^R)$ is Giambelli compatible only if $\P$ is a determinantal point process.
\end{conj}

For further background on the use of symmetric functions in the theory of point processes, see e.g. \cite{infinitewedge}, \cite{okounkovreshetikhin}, \cite{okounkovuses}, \cite{johansson}, \cite{borodinrains}, \cite{borodingorin}, \cite{bufetovgorin}, while for further background on characteristic polynomials of random matrices, one can see \cite{keatingsnaith}, \cite{fyodorovstrahov1}, \cite{fyodorovstrahov2}, \cite{baikdeiftstrahov}, or, more recently, \cite{chhaibinajnudel}, \cite{borodinstrahov}, \cite{giambelli}, \cite{bertolaruzza}, \cite{lazagchristoffel} for the discrete case.

\subsection{The scheme of the argument}
The proof of Theorems \ref{thm1} and \ref{thm2} proceeds as follows. We first develop the notion of uniform $L^1$-specialization for which the Giambelli compatibility is equivalent to the validity of the analogue of formula (\ref{form:charpol}), where the characteristic polynomials are replaced by the specialization of the generating series of the complete homogeneous and elementary symmetric functions, see Proposition \ref{prop:propgiambelli}.

We construct two families of uniform $L^1$-specializations for determinantal point processes with a specific additive structure, see Proposition \ref{prop:l1}, and establish a key lemma in this context, Lemma \ref{lem:mainlem} below. This Lemma allows to derive the Giambelli compatibility for the whole process from the Giambelli compatibility for the point process conditioned to be equal to a specific configuration outside a compact interval $[-T,T]$.

The Lemma is proved in the following way: the conditional expectations of the analogues of characteristic polynomials form a reverse martingale, and satisfy an analogue of formula (\ref{form:charpol}). By \cite{bufetov-qiu-shamov}, Osada-Osada \cite{osada-osada}, Lyons \cite{lyons}, the tail sigma-algebra of a determinantal point process induced by an orthogonal projection is trivial. The additive structure of the specializations leads to the Giambelli compatibility for the whole point process by passing to the limit as $T \rightarrow  + \infty$ in the analogue of formula (\ref{form:charpol}).

It is then easily seen that the specializations of interest, the specializations $\rho^R$ and $\tilde{\rho}^R$ introduced above, fit into this formalism. By Theorems 1.1, 1.4 from \cite{bufetovconditional}, we know that, if $\P$ is a determinantal point process on $E\subset \R$ with an integrable projection kernel satisfying (\ref{cond:hilbertschmidt}), then for $\P$-almost every $X \in \Conf(E)$ and for any $T>0$, the point process in $[-T,T] \cap E$ conditioned to be equal to $X$ outside $[-T,T]$ is an orthogonal polynomial ensemble. We adapt the argument from \cite{giambelli} in order to prove that orthogonal polynomial ensembles are Giambelli compatible for the specializations $\rho^R$ and $\tilde{\rho}^R$, see Proposition \ref{prop:opegiambelli}. We conclude the proof by Lemma \ref{lem:mainlem}.
\subsection{Organization of the paper}
The paper is organized as follows. In Section \ref{sec:ppcond}, we recall the necessary background on point processes and their conditional measures.\\
In Section \ref{sec:det}, we recall the definition of determinantal point processes, both via their additive and multiplicative functionals. We introduce the orthogonal polynomial ensembles and see them as a class of examples of determinantal point processes. We recall the definition of integrable kernels, and state results from \cite{bufetovconditional}  on conditional measures of determinantal point processes.\\
In Section \ref{sec:symfunctions}, following \cite{macdonald}, we recall  the Giambelli formula. The section ends with the statement of a generalized version of the Giambelli formula, together with the  formula \eqref{form:usefull} relating the generating series of the elementary and complete homogeneous symmetric functions to Schur functions indexed by hooks. \\
The  formula \eqref{form:usefull} is used in Section \ref{sec:giambelli1}, where we establish the equivalence of the Giambelli compatibility of a point process and the validity of formula (\ref{form:charpol}) for analogues of characteristic polynomials. We first introduce the notion of uniform $L^1$-specializations for which this equivalence takes place, see Proposition \ref{prop:propgiambelli}. We prove in Proposition \ref{prop:shift} that, for a uniform $L^1$-specialization, the Giambelli compatibility is invariant when one modifies the image of the Newton power sums by adding a constant.\\
We then construct in Section \ref{sec:giambelli2} uniform $L^1$-specializations associated to a sequence of functions for determinantal point processes satisfying condition (\ref{cond:hilbertschmidt}) and governed by an orthogonal projection, see Proposition \ref{prop:l1}. \\
In Section \ref{sec:giambelli3} we state and prove the key Lemma \ref{lem:mainlem}, by means of Propositions \ref{prop:propgiambelli}, \ref{prop:l1}, \ref{prop:shift}, and using the tail triviality.\\
In Section \ref{sec:proof1}, we introduce the specializations $\rho^R$ and $\tilde{\rho}^R$, and see them as particular examples of the above defined formalism. We derive Proposition \ref{prop:1} from the general Proposition \ref{prop:l1} and from Proposition \ref{prop:analyticcont}, saying that the random functions $\Psi_X(u)$ (resp. $\tilde{\Psi}_X$) are analytic continuations of the images by $\rho^R$ (resp. $\tilde{\rho}^R$) of the generating series of the complete homogeneous and elementary symmetric functions. The equivalence of Theorem \ref{thm1} and \ref{thm2} follows now from Proposition \ref{prop:propgiambelli}.\\
We prove in Section \ref{sec:proof2} that orthogonal polynomial ensembles are Giambelli compatible for the specializations $\rho^R$ and $\tilde{\rho}^R$, adapting the argument of Borodin-Olshanski-Strahov (\cite{giambelli}), see Proposition \ref{prop:opegiambelli}. \\
We conclude the proof of Theorems \ref{thm1} and \ref{thm2} in Section \ref{sec:proof3}, by means of Lemma \ref{lem:mainlem}, Propositions \ref{prop:bufetovconditional} and \ref{prop:opegiambelli}.
\subsection*{Acknowledgements} We are deeply grateful to Kristian Bjerkl\"ov, Alexey Klimenko and Yanqi Qiu for helpful discussions and comments. This project has received funding from the European Research Council (ERC) under the European Union's Horizon 2020 research and innovation programme under grant agreement No 647133 (I-CHAOS).
\section{Point processes and conditioning} \label{sec:ppcond}
Let $E$ be a subset of $\R$, either open or countable and without accumulation point. Recall that the space of configurations on $E$, denoted  $\Conf(E)$, is the set of all subsets of $E$ without accumulation points. For a finite set $A$, we write $\#(A)$ for the cardinality of $A$.  We equip the space $\Conf(E)$ with the sigma-algebra $\mathcal{F}_E$ generated by the maps 
\begin{align*}
\begin{matrix} \#_B : \hspace{0.1cm}& \Conf(E) &  \rightarrow & \N \\
& X & \mapsto & \#(X \cap B)
\end{matrix}
\hspace{0.1cm}, \quad \text{$B\subset E$, $B$  is a bounded Borel set.} 
\end{align*}
\begin{defi}\label{def:pp}A point process is a probability measure $\P$ on $(\Conf(E), \mathcal{F}_E)$.
\end{defi}
See also \cite{daley-verejones} for further background on point processes.

We use the notion of {\it rigidity} introduced by Ghosh and Peres. For a subset $W \subset E$,  we denote by $\mathcal{F}_W \subset \mathcal{F}_E$ the sigma-algebra generated by the maps $\#_B$, where $B \subset W$ is a bounded Borel set. We denote by $\mathcal{F}_W^\P$ the completion of $\mathcal{F}_W$ with respect to $\P$.
\begin{defi} [Ghosh-Peres, \cite{ghoshrigid}, \cite{ghoshperesrigid}] \label{def:rigid} A point process is said to be \emph{rigid} if, for any bounded set $B \subset E$, the random variable $\#_B$ is $\mathcal{F}_{B^c}^\P$-measurable.
\end{defi}
We introduce a natural filtration on $\Conf(E)$, and define conditional measures of $\P$. For $T>0$, we write $\mathcal{F}_T^\P:= \mathcal{F}_{[-T,T]^c \cap E}^\P$. The sequence $(\mathcal{F}_T^\P)_{T>0}$ is a reverse filtration, and we denote by $\mathcal{F}_\infty^\P$ the tail sigma-algebra 
\begin{align*}\mathcal{F}_\infty^\P = \cap_{T>0} \mathcal{F}_T^\P.
\end{align*}
For a subset $W \subset E$, we make the identification $\Conf(E) \simeq \Conf(W) \times \Conf(W^c)$. Let $\pi_W$ be the map 
\begin{align*}
\pi_W : \hspace{0.1cm}\Conf(E) &\rightarrow \Conf(W)\\
X &\mapsto X\cap W.
\end{align*}
For a point process $\P$ and $(\pi_W)_*(\P)$-almost every configuration $X \in \Conf(W)$, we define the measure $\P(\cdot | X, W)$ supported on $\{X\} \times \Conf(W^c)$ by the disintegration formula 
\begin{align*}
\P(\mathcal{B})=\int_{\Conf(W)} \P(\mathcal{B} |X,W) (\pi_W)_*(\P)(dX), \quad \mathcal{B}\subset \Conf(E) \text{ is a Borel set.}
\end{align*}
The measure $\P(\cdot | X,W)$ is the conditional measure of $\P$ on $W^c$ under the condition that the configuration coincide with $X$ in $W$. For $\P$-almost every configuration $X \in \Conf(E)$ we write
\[ \P(\cdot | X, W):=  \P(\cdot | X\cap W, W),
\]
and for $T>0$, we will write 
\begin{align*}
\P(\cdot |X, T):= \P(\cdot |X, [-T,T]^c).
\end{align*}
For an integrable random variable $F : \Conf(E) \rightarrow \C$, and $\P$-almost every configuration $X \in \Conf(E)$, we denote by 
\[\E_\P[ F | X,T] = \E_\P[F | \mathcal{F}_T^\P](X)
\]
its expectation with respect to $\P( \cdot |X,T)$. 

\section{Determinantal point processes} \label{sec:det}
We now define determinantal point processes, first via their additive functionals (\ref{def:det1}), and later with their multiplicative functionals (\ref{def:det2}). We refer to \cite{simontrace} for the notion of trace class operators and their Fredholm determinants. Let $\mu$ be a Radon measure on a subset $E$ of $\R$. We consider here the continuous case when the set $E$ is open and $\mu$ has no atoms, or the discrete one when the set $E$ is countable and has no accumulation point and $\mu$ is the counting measure on $E$.
\begin{defi} \label{def:det1}A point process $\P$ on $E$ is called a determinantal point process on $(E,\mu)$ if there exists a function 
\begin{align*}
K : \hspace{0.1cm} E \times E \rightarrow \C,
\end{align*}
called a correlation kernel for $\P$, such that for all $n \in \N$ and all bounded Borel compactly supported functions $f : E^n \rightarrow \C$, one has 
\begin{align*}
\E_\P \overset{*}{\sum} f(x_1,...x_n) = \int_{E^n}f(x_1,...,x_n)\det \left( K(x_i,x_j) \right)_{i,j=1}^n d\mu(x_1)...d\mu(x_n),
\end{align*}
where $\overset{*}{\sum}$ denotes the sum over all pairwise distinct points $x_1,...,x_n$ in the random configuration $X \in \Conf(E)$.
\end{defi}
We assume that the kernel $K$ induces a locally trace-class integral operator acting in  $L^2(E,d\mu)$ and for  which we keep the same letter $K$ 
\begin{align*}
K : L^2(E,d\mu) &\rightarrow L^2(E,d\mu) \\
f &\mapsto \left( x \mapsto \int_E f(y) K(x,y) d\mu(y) \right).
\end{align*}
For any bounded Borel set $B \subset E$, we have
\[\text{Tr}(\chi_B K \chi_B) = \int_B K(x,x)d\mu(x) \hspace{0.1cm}= \E_\P[\#_B].
\]  Conversely, by the Macchi-Soshnikov/Shirai-Takahashi Theorem  (\cite{macchi}, \cite{soshnikov}, \cite{shiraitakahashi}), a Hermitian locally trace class operator gives rise to a determinantal point process if and only if it is a positive contraction. In particular, any locally trace class orthogonal projection gives rise to a determinantal point process. An equivalent definition to \ref{def:det1} is the following 
\begin{defi} \label{def:det2}A point process $\P$ is a determinantal point process if there exists a kernel $K$ such that for any function $g : E \rightarrow \C$ such that $g-1$ is compactly supported, one has 
\begin{align*}
\E_\P \left[ \prod_{x \in X} g(x)\right] = \det \left(I + \chi_B K (g-1) \right),
\end{align*}
where $B \subset E$ is a bounded Borel set containing the support of $g-1$, and $\det$ denotes the Fredholm determinant in $L^2(E,d\mu)$ of the trace class operator $\chi_B K (g-1)$.
\end{defi}
A first class of examples of determinantal point processes we will be interested in is the one of orthogonal polynomial ensembles, which have a fixed number of particles.
\begin{defi}Let $N \in \N$ be a positive integer. An $N$-orthogonal polynomial ensemble is a probability measure $\P_N$ on $E^N$ defined by 
\begin{align*}
d\P_N(x_1,...,x_N)=C_N\prod_{1 \leq i<j \leq N}(x_i-x_j)^2 \prod_{i=1}^Nd\omega(x_i),
\end{align*}
where $\omega$ is a measure on $E$ with finite moments of all orders.
\end{defi}
Setting 
\begin{align*}
m_N:= \det \left( \int x^{i+j} d\omega(x) \right)_{i,j=0}^{N-1},
\end{align*}
for the normalizing constant we have
\begin{align*}
C_N^{-1}= N!m_N.
\end{align*}
Below we need a slightly more general
\begin{prop}\label{prop:normalizingconstant} Let $\omega$ be a measure on $E$ with finite moments of all orders. For $a \in \C$ and $N \in \N$, set :
\begin{align*}
m_N(a)= \det \left( \int_\R (x+a)^{i+j} d\omega(x) \right)_{i,j=0}^{N-1}.
\end{align*}
Then  for all $a \in \C$ we have
\begin{align*}
\int_{E^N} \prod_{1 \leq i <j \leq N} (x_i-x_j)^2 d\omega(x_1)\dots d\omega(x_N) = N! m_N(a).
\end{align*}
In particular, $m_N(a)=m_N$ does not depend on $a$.
\end{prop}
\begin{proof}
We clearly have 
\begin{multline*}
\int_{E^N} \prod_{1 \leq i <j \leq N} (x_i-x_j)^2 d\omega(x_1)\dots d\omega(x_N) \\
= \int_{E^N} \prod_{1 \leq i<j \leq N} (x_i +a -(x_j +a))^2 d\omega(x_1)\dots d\omega(x_N) \\
=\int_{E^N} \det \left( (x_i+a)^{j-1} \right)_{i,j=1}^N \det \left( (x_i+a)^{j-1} \right)_{i,j=1}^N d\omega(x_1)\dots d\omega(x_N).
\end{multline*}
The Andr\'eief Formula \cite{andreief} yields the equality
\begin{multline*}
\int_{E^N} \det \left( (x_i+a)^{j-1} \right)_{i,j=1}^N \det \left( (x_i+a)^{j-1} \right)_{i,j=1}^N d\omega(x_1)\dots d\omega(x_N) \\
= N! \det \left( \int_E (x+a)^{i+j} d\omega(x) \right)_{i,j=0}^{N-1}= N! m_N(a),
\end{multline*}
and the proof is complete. 
\end{proof}
\begin{rem}It would be interesting to have a direct proof of the fact that $m_N(a)$ does not depend on $a$.
\end{rem}
Since an orthogonal polynomial ensemble as defined above is a symmetric measure on $E^N$ which does not charge any $N$-tuple with equal coordinates, it is identified with a point process on $E$. This point process is determinantal (see e.g. \cite{konig} and references therein, or \cite{baikdeiftsuidan}) :
\begin{prop}An $N$-orthogonal polynomial ensemble is a determinantal point process on $(E,\omega)$. Its kernel is the projection kernel onto the space of polynomials of degree less than $N$.
\end{prop}

Recall statements 2 from Theorems 1.1 and 1.4 in \cite{bufetovconditional}
\begin{prop} \label{prop:bufetovconditional}
Let $\P$ be a determinantal point process governed by an orthogonal projection, and such that its correlation kernel has an integrable form and satisfy the growth condition (\ref{cond:hilbertschmidt}). Assume that $\P$ is rigid in the sense of Definition \ref{def:rigid}, and let $T>0$. Then, for $\P$-almost every $X \in \Conf(E)$, the point process $\P(\cdot | X, T)$ is an orthogonal polynomial ensemble.
\end{prop}

\section{The Giambelli Formula } \label{sec:symfunctions}
 
\subsection{Power sums, complete homogeneous functions and elementary symmetric functions}
Here we  recall  from   Macdonald \cite{macdonald} the properties of symmetric functions that we will use in Section \ref{sec:giambelli}.
Let $\Lambda$ be the algebra over $\C$ of symmetric functions in the formal variables $\x_1,\x_2,\dots$, see \cite{macdonald}. A distinguished basis in $\Lambda$ is provided by the Newton power sums 
\begin{align*}
\p_k(\x_1,\x_2,\dots)= \sum_{i=1}^{+ \infty} \x_i^k, \quad k=1,2,\dots.
\end{align*}
The elementary symmetric functions $\{\e_k\}_{k \geq 0}$ and the complete homogeneous function $\{\h_k\}_{k \geq 0}$ are defined as follows 
\begin{align*}
\e_k = \sum_{i_1 <\dots <i_k} \x_{i_1} \dots \x_{i_k}, \quad \h_k = \sum_{i_1 \leq \dots \leq i_k} \x_{i_1} \dots \x_{i_k}, \quad k=0,1,\dots 
\end{align*}
In particular, $\h_0=\e_0 =1$, and we agree that 
\begin{align*}
\h_{-1}=\h_{-2}= \dots =\e_{-1}= \e_{-2}=\dots = 0. 
\end{align*}
The following identities hold for the generating series of the three above-defined families of symmetric functions 
\begin{align}\label{formula:genseriesH}
H(\uuu)&:=\sum_{k=0}^{+\infty}\h_k \uuu^k = \exp \left( \sum_{k=1}^{+\infty} \p_k\frac{\uuu^k}{k} \right),\\ \label{formula:genseriesE}
E(\uuu)&:= \sum_{k=0}^{+\infty} \e_k \uuu^k = \left(H(-\uuu)\right)^{-1},
\end{align}
and also 
\begin{align} \label{formula:genseriesHE}
H(\uuu)=\prod_{i=1}^{+\infty}(1-\uuu\x_i)^{-1}, \quad E(\uuu)= \prod_{i=1}^{+ \infty}(1+ \uuu\x_i).
\end{align}
\subsection{Young diagrams, Schur functions and the Giambelli formula}
Another basis in $\Lambda$ is formed by the Schur functions, which are indexed by Young diagrams. A Young diagram $\lambda=(\lambda_1 \geq \lambda_2 \geq \dots)$ is an almost zero non-increasing sequence of non-negative integers. The set of all Young diagrams is denoted by $\Y$, and for $\lambda \in \Y$, we let $l(\lambda)$ be the index of its last non-zero entry, and we write  
\begin{align*}
|\lambda |= \sum_{i \geq 1} \lambda_i.
\end{align*}
The transposed diagram $\lambda'$ of a Young diagram $\lambda \in \Y$ is defined by 
\begin{align*}
\lambda_i'= |\{j, \hspace{0.1cm} \lambda_j \geq i \} |,\quad i=1,2,\dots.
\end{align*}
A Young diagram can also be written using the Frobenius notation. For $\lambda= (\lambda_1 \geq \lambda_2,\dots) \in \Y$, let $d$ be the number of diagonal boxes in $\lambda$, i.e. $d=\max \{i =1,2,\dots,\hspace{0.1cm} \lambda_i \geq i \}$, and for $i=1,\dots, d$, the Frobenius coordinates $p_i,q_i$ are defined by 
\begin{align*}
p_i = \lambda_i-i, \quad q_i=\lambda_i'-i.
\end{align*}
The Frobenius notation for $\lambda$ is then $\lambda = (p_1,\dots,p_d | q_1,\dots,q_d)$. In particular, for $p,q  \in \N$, $(p|q)$ is the Frobenius notation for the Young diagram $(p+1,1^q)$.\\
\par
The Schur functions can be defined by the Jacobi-Trudi formula
\begin{align*}
\s_\lambda = \det \left(\h_{\lambda_i-i+j} \right)_{i,j=1}^n = \det \left( \e_{\lambda_i'-i+j} \right)_{i,j=1}^n,
\end{align*}
valid under the assumption $n \geq l(\lambda)$ for the first equality, $n \geq l(\lambda^{\prime})$ for the second. In particular, we have 
\begin{align*}
\s_{(r)}=\h_r, \quad \s_{(1^r)}=\e_r.
\end{align*}
Equivalently, Schur functions can be defined as the projective limits in $\Lambda$ of the Schur polynomials 
\begin{align} \label{def:schurpol}
\s_\lambda(\x_1,\dots,\x_n) = \frac{\det ( \x_i^{\lambda_j-j+n})_{i,j=1}^n}{\prod_{1\leq i <j \leq n}(\x_i-\x_j)},
\end{align}
where $n \geq l(\lambda)$. The Schur functions satisfy the Giambelli formula (\cite{macdonald}, I.3 ex.9): if, in the Frobenius notation, $\lambda=(p_1,\dots,p_d |q_1,\dots, q_d)$ then 
\begin{align} \label{form:giambelli}
\s_\lambda = \det (\s_{(p_i|q_j)} )_{i,j=1}^d.
\end{align}
A more general statement holds, see \cite{macdonald}, ch. I.3 ex. 21.
\begin{prop} \label{prop:giambelligeneral} Let $\{ \h_{r,s}\}_{r \in \Z, \hspace{0.1cm} s \geq 0}$ be a collection of commuting indeterminates such that 
\begin{align*}
\h_{0,s}=1 \hspace{0.1cm} \text{for all } s \geq 0; \quad \h_{r,s} = 0 \hspace{0.1cm} \text{for all } r <0 , s \geq 0.
\end{align*}
For $\lambda \in \Y$ and $N \geq l(\lambda)$, define 
\begin{align*}
\tilde{\s}_{\lambda} = \det ( \h_{\lambda_i-i+j, j-1} )_{i,j =1}^N.
\end{align*}
We then have, if $\lambda = (p_1, \dots, p_d | q_1, \dots , q_d)$, 
\begin{align*}
\tilde{\s}_\lambda = \det (  \tilde{\s}_{(p_i|q_j)} )_{i,j=1}^d. 
\end{align*}
\end{prop}
Another useful formula is the following (\cite{macdonald}, I.3, ex. 14) :
\begin{equation} \label{form:usefull}
H(\uuu)E(\vvv) = 1 + (\uuu+\vvv) \sum_{p,q=0}^{+\infty} \s_{(p|q)} \uuu^p \vvv^q.
\end{equation}
\begin{rem}\label{rem:formusefull}Both identities (\ref{form:giambelli}) and (\ref{form:usefull}) follow from the fact that 
\begin{equation*}
\s_{(p|q)} + \s_{(p+1|q-1)} = \h_{p+1} \e_q,
\end{equation*}
which itself is a direct consequence of Jacobi-Trudi formula.
\end{rem}
\section{Specializations of the algebra of symmetric functions and Giambelli compatible point processes} \label{sec:giambelli}
We introduce in Section \ref{sec:giambelli1} the notion of uniform $L^1$-specializations of the algebra $\Lambda$ for a given point process, and prove that the Giambelli compatiblity is equivalent to the validity of formula (\ref{form:charpol}), where the product 
\begin{align*}
\rho(E(z))\rho(H(w)),
\end{align*}
plays the r\^ole of the ratio $P(z)/P(w)$ of characteristic polynomials,see Proposition \ref{prop:propgiambelli}. In Section \ref{sec:giambelli2}, we construct two classes of uniform $L^1$-specializations for determinantal point processes. We finally prove in Section \ref{sec:giambelli3} that the Giambelli compatibility of a determinantal point process with a uniform $L^1$-specialization follows from the Giambelli compatibility of the conditioned point processes.
\subsection{Uniform $L^1$-specializations} \label{sec:giambelli1}
\begin{defi}Let $\rho : \Lambda \rightarrow L^1(\Conf(\R), \P)$ be a specialization, and let $R>0$. We say that the the specialization $\rho$ is a \emph{$R$-uniform $L^1$-specialization} if for $\P$-almost every configuration $X \in \Conf(E)$, the radius of convergence of the series 
\begin{align} \label{eq:defS}
S_X^\rho(u):=\sum_{k \geq 1} \frac{\rho(\p_k)(X)}{k}u^k
\end{align}
is at least $R$ and if furthermore we have
\begin{align} \label{cond:unifl1} \exp\left(\sum_{k=1}^nS_X^\rho(w_k)-S_X^\rho(z_k)\right) \in L^1(\Conf(E), \P),
\end{align}
for all $n \in \N$ and all $z_1,\dots,z_n,w_1,\dots, w_n \in D(0,R)$ where $D(0,R)$ is the open disk centered at the origin and of radius $R$: $D(0,R)=\{z \in \C, \hspace{0.1cm} |z| <R\}$.
\end{defi}
\begin{rem}The requirement that the radius of convergence of the series $S_X^\rho(u)$ is at least $R>0$ independently of the configuration $X \in \Conf(E)$ is stronger than that the series $S_X^\rho(u)$ have a positive radius of convergence for $\P$-almost every $X \in \Conf(E)$.
\end{rem}
Observe that, if $\rho$ is a $R$-uniform $L^1$-specialization, formulas (\ref{formula:genseriesH}) and (\ref{formula:genseriesE}) specialize into
\begin{align}
\exp \left( S_X^\rho(w) \right) &= \sum_{k =0}^{+ \infty} \rho(\h_k)(X) w^k = \rho(H)(w), \\
\exp \left( -S_X^\rho(z) \right) &= \sum_{k=0}^{+\infty} (-1)^k\rho(\e_k)(X) z^k = \rho(E)(-z),
\end{align}
the series having radius of convergence at least $R>0$. We also have from (\ref{form:usefull}) that 
\begin{align} \label{form:usefullspec}
\exp \left( S_X^\rho(w)-S_X^\rho(z) \right) = \rho(H)(w)\rho(E)(-z) = 1 + (w-z)\sum_{p,q=0}^{+\infty} (-1)^q \rho( \s_{(p|q)})(X) w^p z^q,
\end{align}
for all $z,w \in D(0, R)$. These observations and especially the latter formula (\ref{form:usefullspec}) lead to the following Proposition, analogous to Proposition 2.2. in \cite{giambelli}.
\begin{prop} \label{prop:propgiambelli} Let $\rho$ be a $R$-uniform $L^1$-specialization. Then $\rho$ is an $L^1$-specialization and $(\P,\rho)$ is Giambelli compatible if and only if, for all $n \in \N$ and all pairwise distinct $z_1,\dots,z_n,w_1,\dots,w_n \in D(0,R)$, we have 
\begin{multline} \label{eq:propgiambelli'}
\E_\P \left[ \exp \left(\sum_{i=1}^n S_X^\rho(w_i)-S_X^\rho(z_i)\right) \right]\\
 = \det \left( \frac{1}{z_i-w_j} \right)_{1 \leq i,j\leq n}^{-1} \det \left( \frac{1}{z_i-w_j}\E_\P \left[ \exp (S_X^\rho(w_i)-S_X^\rho(z_j) \right] \right)_{i,j=1}^n.
\end{multline}
\end{prop}
\begin{proof}
Let $z_1,\dots , z_n, w_1 \dots w_n \in D(0,R)$ be pairwise distinct. From formulas (\ref{formula:genseriesH}) and (\ref{formula:genseriesE}), equation (\ref{eq:propgiambelli'}) can be written as 
\begin{multline}\label{eq:propgiambelli}
\E_\P \left[ \prod_{i=1}^n \rho(H)(w_i)\rho(E)(-z_i) \right] \\
= \det \left( \frac{1}{z_i-w_j} \right)_{1 \leq i,j\leq n}^{-1} \det \left( \frac{1}{z_i-w_j}\E_\P \left[ \rho(H)(w_i)\rho(E)(z_j) \right] \right)_{i,j=1}^n.
\end{multline}
The proof we present is adapted from \cite{giambelli}, Proposition 2.2, and is based on identity (\ref{form:usefullspec}). For brevity, we omit the specialization $\rho$ and write simply 
\begin{align*}
f(X)&:=\rho(f)(X), \quad f \in \Lambda, \quad X \in \Conf(E),\\
F(z)&:= \rho(F)(z), \quad F \in \Lambda [[z]].
\end{align*}
By formula (\ref{form:usefullspec}), we have for $\P$-a.e. $X \in \Conf(\R)$ 
\begin{multline} \label{eq:propgiambellicauchy}
\det \left( \frac{ H(w_j)E(-z_i) - 1}{z_i-w_j} \right)_{i,j=1}^n = \det \left( \sum_{p_j,q_i=0}^{+\infty} (-1)^{q_i}\s_{(p_j|q_i)}(X)w_j^{p_j}z_i^{q_i} \right)_{i,j=1}^n \\
= \sum_{p_1,\dots, p_n,q_1, \dots, q_n=0}^{+\infty} (-1)^{q_1+ \dots + q_n} \s_{(p_1, \dots, p_n |q_1,\dots,q_n)}(X)w_1^{p_1}\dots w_k^{p_n}z_1^{q_1} \dots z_k^{q_n}.
\end{multline}
Cauchy's contour integral formula now implies that $\rho$ is an $L^1$-specialization, as the left-hand-side of equation (\ref{eq:propgiambellicauchy}) is integrable with respect to $\P$ by condition (\ref{cond:unifl1}). Assume now that $\P$ is Giambelli compatible in the sense of Definition \ref{def:giambelli}.
Taking expectation with respect to $\P$ on both sides of the latter equality (\ref{eq:propgiambellicauchy}), we have 
\begin{multline} \label{eq:giambelli0.25} 
\E_\P \left[ \det \left( \frac{ H(w_j)E(-z_i) - 1}{z_i-w_j} \right)_{i,j=1}^n \right]
\\
= \sum_{p_1,\dots, p_n,q_1, \dots, q_n=0}^{+\infty} (-1)^{q_1+ \dots + q_n}  \E_\P [ \s_{(p_1, \dots, p_n |q_1,\dots,q_n)}(X)]w_1^{p_1}\dots w_n^{p_n}z_1^{q_1} \dots z_n^{q_n}, 
\end{multline}
where the interchange of summation and expectation is justified from the local uniform convergence of the series $H(u)$ and $E(u)$ on $D(0,R)$. Since $\P$ is Giambelli compatible, we obtain 
\begin{multline}
\E_\P \left[ \det \left( \frac{ H(w_j)E(-z_i) - 1}{z_i-w_j} \right)_{i,j=1}^n \right]\\
= \sum_{p_1,\dots, p_n,q_1, \dots, q_n=0}^{+\infty} (-1)^{q_1+ \dots + q_n} { \det \left( \E_\P [\s_{(p_i |q_j)(X)}] \right)_{i,j=1}^n}{w_1^{p_1}\dots w_n^{p_n}z_1^{q_1} \dots z_n^{q_n}},
\end{multline}
which again might be written as 
\begin{multline} \label{eq:giambelli0.5} 
\E_\P \left[ \det \left( \frac{ H(w_j)E(-z_i) - 1}{z_i-w_j} \right)_{i,j=1}^n \right] \\
=\det \left( \sum_{p_j,q_i=0}^{+\infty} (-1)^{q_i}\E_\P [{\s_{(p_j|q_i)}(X)]}{w_j^{p_j}z_i^{q_i}} \right)_{i,j=1}^n  \\
= \det \left(  \E_\P \left[ \frac{ H(w_j)E(-z_i) - 1}{z_i-w_j} \right] \right)_{i,j=1}^n.
\end{multline}
We now want to get rid of the $-1$. Consider the $n \times n$ matrices
\[A=(A(i,j))_{i,j=1}^n= \left( \frac{ H(w_j^{-1})E(-z_i^{-1}) - 1}{z_i-w_j} \right)_{i,j=1}^n, \quad B=(B(i,j))_{i,j=1}^n= \left( \frac{1}{z_i-w_j} \right)_{i,j=1}^n,\]
and for subsets $I,J \subset \{1, \dots, n\}$ with $|I|=|J|$, denote by $A_{I,J}$ (resp. $B_{I,J}$) the submatrix of $A$ (resp. $B$) with entries indexed by $I \times J$. The preceding argument can be repeated in order to establish that 
\[ \E_\P [ \det (A_{I,J}) ] = \det ( \E_\P [A(i,j)] )_{i \in I, j \in J}. \]
From the expansion 
\begin{align*}
\det ( A + B) = \sum_{I, J \subset \{1, \dots, n\}, |I|=|J|} \det(A_{I,J}) \det(B_{I^c,J^c}),
\end{align*}
since $B$ is deterministic, we have 
\begin{multline*}
 \E_\P [ \det (A + B) ]  = \sum_{I, J \subset \{1, \dots, n \}, |I|=|J|}  \det( B_{I^c, J^c}) \E_\P [\det(A_{I,J})] \\
 =  \sum_{I, J \subset \{1, \dots, n\}, |I|=|J|}  \det(B_{I^c,J^c}) \det ( \E_\P [A(i,j)] )_{i \in I, j \in J} = \det( \E_\P [ A+B ]),
\end{multline*}
i.e. 
\begin{align*}
 \E_\P \left[ \det \left( \frac{ H(w_j)E(-z_i) }{z_i-w_j} \right)_{i,j=1}^n \right] =   \det \left( \E_\P \left[ \frac{ H(w_j)E(-z_i) }{z_i-w_j} \right] \right)_{i,j=1}^n.
\end{align*}
Using the multiplicativity of the determinant, the latter equality can be rewritten as 
\begin{align} \label{eq:giambelli1}
\E_\P \left[ \prod_{j=1}^n H(w_j)E(-z_j) \right] = \det \left( \frac{1}{z_i-w_j} \right)^{-1}_{1 \leq i,j \leq k} \det \left( \frac{1}{z_i-w_j} \E_\P \left[ H(w_j) E(z_i)\right] \right)_{i,j=1}^n.
\end{align}
The ‘‘only if’’ part of the proposition is established.
\par
For the ‘‘if’’ part, observe that all the arguments above can be reversed. Indeed, (\ref{eq:propgiambelli}) is equivalent to (\ref{eq:giambelli1}), which is equivalent to (\ref{eq:giambelli0.5}). One obtains the Giambelli compatibility (\ref{eq:giambelli0}) by equating the coefficients in the expansion performed in (\ref{eq:giambelli0.25}). The proof is complete.
\end{proof}
The structure of equation (\ref{eq:propgiambelli'}) implies the following proposition, which we will use later in Lemma \ref{lem:mainlem}.
\begin{prop}\label{prop:shift}Let $\rho$ be a $R$-uniform $L^1$-specialization and let $a=(a_k)_{k \geq 1} \subset \C$ be a bounded sequence of complex numbers. We define the specialization $\rho_a$ by 
\begin{align*}
\rho_a(\p_k)=\rho(\p_k)-a_k, \quad k=1,2,\dots
\end{align*}
Assume that there exists $R'>0$ such that $\rho_a$ is a $R'$-uniform $L^1$-specialization. Then $(\P, \rho)$ is Giambelli compatible if and only if $(\P,\rho_a)$ is Giambelli compatible.
\end{prop}
\begin{proof}
From definition of the series $S_X^\rho(u)$ (\ref{eq:defS}), we have for all $z,w \in D(0,\min(R,R',1))$
\begin{align*}
S_X^{\rho_a}(w) - S_X^{\rho_a}(z) = S_X^{\rho}(w) - S_X^\rho(z) + \sum_{k \geq 1} \frac{a_k}{k}(z^k-w^k).
\end{align*}
Thus, on the one hand, one can write
\begin{multline*}
\E_\P \left[ \exp \left( \sum_{i=1}^n S_X^{\rho_a}(w_i)- S_X^{\rho_a}(z_i) \right) \right] \\
= \exp \left( \sum_{i=1}^n \sum_{k \geq 1} \frac{a_k}{k}(w_i^k-z_i^k) \right) \E_\P \left[ \exp \left( \sum_{i=1}^n S_X^{\rho}(w_i)- S_X^{\rho}(z_i) \right) \right],
\end{multline*}
for all pairwise distinct $z_1,\dots,z_n,w_1,\dots, w_n \in D(0,\min(R,R',1))$. On the other hand, the multilinearity of the determinant implies that
\begin{multline*}
\det \left( \frac{1}{z_i-w_j} \E_\P \left[ \exp \left( S_X^{\rho_a}(w_j) - S_X^{\rho_a}(z_i) \right) \right] \right)_{i,j=1}^n \\
= \exp \left( \sum_{i=1}^n \sum_{k \geq 1} \frac{a_k}{k}(w_i^k-z_i^k) \right) \det \left( \frac{1}{z_i-w_j} \E_\P \left[ \exp \left( S_X^{\rho}(w_j) - S_X^{\rho}(z_i) \right) \right] \right)_{i,j=1}^n.
\end{multline*}
Thus, the factor
\[ \exp \left( \sum_{i=1}^n \sum_{k \geq 1} \frac{a_k}{k}(w_i^k-z_i^k) \right) \]
can be cancelled out on both sides of the equation
\begin{multline*}
\E_\P \left[ \exp \left( \sum_{i=1}^n S_X^{\rho_a}(w_i)- S_X^{\rho_a}(z_i) \right) \right] \\
= \det \left( \frac{1}{z_i-w_j} \right)_{1 \leq i,j \leq n}^{-1} \det \left( \frac{1}{z_i-w_j} \E_\P \left[ \exp \left( S_X^{\rho_a}(w_j) - S_X^{\rho_a}(z_i) \right) \right] \right)_{i,j=1}^n
\end{multline*}
and one obtains equation (\ref{eq:propgiambelli'}). Recalling Proposition \ref{prop:propgiambelli}, the proof is complete.
\end{proof}
\subsection{Uniform $L^1$-specializations for determinantal point processes} \label{sec:giambelli2}
Let $E$ be a subset of $\R$, $\mu$ a Radon measure on $E$ and $\P$ a determinantal point process on $(E,\mu)$ with correlation kernel $K$. We assume that the kernel $K$ is the kernel of an orthogonal projection. Let $(f_k)_{k \geq 1}$ be a sequence of bounded functions 
\[f_k : E \rightarrow \C \]
such that there exists $R>0$ such that
\begin{align} \label{cond:cvf}
\sup_{k \geq 1} R^k ||f_k||_{L^\infty(E,d\mu)} <+\infty.
\end{align}
The latter condition implies that the radius of convergence of the series
\begin{align*}
f_u(x):=\sum_{k \geq 1} \frac{f_k(x)}{k} u^k
\end{align*}
is at least $R$, for $\mu$-almost every $x \in E$. We assume that
\begin{align} \label{cond:cvf1}
\int_{E} |f_1(x)|^2K(x,x) d\mu(x) &<+\infty,\\ \label{cond:cvfk}
\int_E |f_k(x)| K(x,x) d\mu(x) &<+\infty, \quad k \geq 2,
\end{align}
and that there exists a function $h : E \rightarrow [0,+ \infty)$ measurable with respect to the measure $\mu$ verifying
\begin{align} \label{cond:h1}
\int_E h(x) K(x,x) d\mu(x) <+\infty
\end{align}
and such that
\begin{align} \label{cond:h2}
\left| f_u(x) - uf_1(x) \right| = O \left(h(x) \right)
\end{align}
as $|x| \rightarrow + \infty$, uniformly in $u$ on compact sets. The last assumption we need is
\begin{align}  \label{cond:cvexpf}
\int_E \left|\exp \left(\sum_{j=1}^lf_{w_j}(x)-\sum_{j=1}^{l'}f_{z_j}(x)\right)-1\right|^2 K(x,x) d\mu(x) &<+\infty,
\end{align}
for all $l,l' \in \N$ and all $z_1,\dots,z_l,w_1,\dots,w_{l'} \in D(0,R)$.

The main result of this section is the following Proposition.
\begin{prop} \label{prop:l1}The specializations $\rho$ and $\tilde{\rho}$ defined by
\begin{align*}
\rho(\p_k):= \sum_{x \in X} f_k(x) - \E_\P \sum_{x \in X} f_k(x), \quad k=1,2,\dots
\end{align*}
and
\begin{align*}
\tilde{\rho}(\p_1)= \rho(\p_1), \quad\tilde{\rho}(\p_k)=\sum_{x \in X} f_k(x), \quad k \geq 2,
\end{align*}
are $R$-uniform $L^1$-specializations.
\end{prop}
The proof requires some steps. The following lemma guarantees the existence of the above-defined specializations $\rho$ and $\tilde{\rho}$.
\begin{lem} \label{lem:cvf} Under condition (\ref{cond:cvf1}), there exists an increasing sequence $(T_N)_{N \in \N}$ tending to $+\infty$ as $N \rightarrow + \infty$ such that the limit
\begin{align}
\lim_{N \rightarrow + \infty} &\sum_{x \in X\cap [-T_N,T_N]} f_1(x) - \E_\P \sum_{x \in X \cap [-T_N,
T_N]} f_1(x), \quad k=1,2,\dots
\end{align}
exists $\P$-almost surely and in $L^1(\Conf(E),\P)$. Condition (\ref{cond:cvfk}) implies that, for all $k \geq 2$, the sums
\begin{align}
\sum_{x \in X} f_k(x)
\end{align}
are absolutely convergent for $\P$-almost every $X \in \Conf(E)$ and belong to $L^1(\Conf(E),\P)$.
\end{lem}
\begin{proof}
By condition (\ref{cond:cvf1}), the variances of the sums
\begin{align*}
\sum_{x \in X} f_1(x)
\end{align*}
is finite. By \cite[Proposition 4.1]{bufetovquasisymmetries}, the random variable
\begin{align*}
\sum_{x \in X} f_1(x) - \E_\P \sum_{x \in X} f_1(x)
\end{align*}
is a well-defined element of $L^2(\Conf(E),\P) \subset L^1(\Conf(E),\P)$. The first statement of the Lemma follows. Observe now that, for all $k \geq 2$, we have
\begin{align*}
\E_\P \sum_{x \in X} |f_k(x)| = \int_E |f_k(x)| K(x,x) d\mu(x),
\end{align*}
which establishes the second statement.
\end{proof}
\begin{lem} \label{lem:cvunif} Under condition (\ref{cond:cvexpf}), the limit 
\begin{align*}
\lim_{T \rightarrow + \infty} \sum_{x \in X \cap [-T,T]} f_u(x) - \E_\P \sum_{x \in X \cap [-T,T]} f_u(x) 
\end{align*}
exists in $L^1(\Conf(E), \P)$. If there exists a function $h:E \rightarrow [0, + \infty)$ satisfying conditions (\ref{cond:h1}) and (\ref{cond:h2}), then the limit
\begin{align*}
\lim_{T \rightarrow + \infty} \sum_{x \in X \cap [-T,T]} f_u(x) - \E_\P \sum_{x \in X \cap [-T,T]} uf_1(x) 
\end{align*}
exists in $L^1(\Conf(E),\P)$. Moreover, there exists a Borel set $\mathcal{W} \subset \Conf(E)$ such that, for any $r \in [0,R)$, there exists a sequence $(T_N)_{N \in \N}$ tending to $+ \infty$ as $N \rightarrow + \infty$ such that the limits
\begin{align*}
\lim_{N \rightarrow + \infty} \sum_{x \in X \cap [-T_N,T_N]} f_u(x) - \E_\P \sum_{x \in X \cap [-T_N,T_N]} f_u(x), \\
 \lim_{N \rightarrow + \infty} \sum_{x \in X \cap [-T_N,T_N]} f_u(x) - \E_\P \sum_{x \in X \cap [-T_N,T_N]} uf_1(x)
\end{align*}
exist for all $X \in \mathcal{W}$, uniformly for $u \in D(0,r)$.
\end{lem}
\begin{proof}Since, for any fixed $u \in D(0,R)$, the function $f_u(x)$ is bounded, the function
\begin{align*}
\left|\frac{\exp\left(f_u(x)\right)-1}{f_u(x)} \right|
\end{align*}
is bounded away from $0$ and $+ \infty$. Thus, condition (\ref{cond:cvexpf}) implies that
\begin{align*}
\int_E |f_u(x) |^2 K(x,x) d\mu(x) <+\infty.
\end{align*}
The variance of the sum
\begin{align*}
\sum_{x \in X} f_u(x)
\end{align*}
is then finite. By \cite[Proposition 4.1]{bufetovquasisymmetries}, the limit
\begin{align} \label{cvl1fu}
\lim_{T \rightarrow + \infty} &\sum_{x \in X \cap [-T,T]} f_u(x) - \E_\P \sum_{x \in X \cap [-T,T]} f_u(x)
\end{align}
exists in $L^1(\Conf(E),\P)$. Observe now that, for any $T>0$, we have
\begin{multline*}
\sum_{x \in X \cap [-T,T]} f_u(x) - \E_\P \sum_{x \in X \cap [-T,T]} uf_1(x) \\
= \sum_{x \in X \cap [-T,T]} f_u(x) - \E_\P \sum_{x \in X \cap [-T,T]} \left(f_u(x) -f_u(x) + uf_1(x) \right)\\
= \sum_{x \in X \cap [-T,T]} f_u(x) - \E_\P \sum_{x \in X \cap [-T,T]} f_u(x) + \int_{-T}^T \left( f_u(x) - uf_1(x) \right) K(x,x) d\mu(x).
\end{multline*}
If there exists a function $h : E \rightarrow [0,+ \infty)$ satisfying (\ref{cond:h1}) and (\ref{cond:h2}), the integral
\begin{align*}
\int_E \left( f_u(x) - uf_1(x) \right) K(x,x) d\mu(x)
\end{align*}
is absolutely convergent, whence the limit
\begin{align*}
\lim_{T \rightarrow + \infty} \sum_{x \in X \cap [-T,T]} f_u(x) - \E \sum_{x \in X \cap [-T,T]} uf_1(x) 
\end{align*}
exists in $L^1(\Conf(E),\P)$. As a consequence of the convergence (\ref{cvl1fu}), there exists a Borel set $\mathcal{W} \subset \Conf(E)$ satisfying $\P(\mathcal{W})=1$ and a sequence $(T_N)_{N \in \N}$ such that
the limits
\begin{align} \label{cv:psfu}
\lim_{N \rightarrow + \infty} &\sum_{x \in X \cap [-T_N,T_N]} f_u(x) - \E_\P \sum_{x \in X \cap [-T_N,T_N]} f_u(x) \\ \label{cv:psf1}
\lim_{N \rightarrow + \infty} &\sum_{x \in X \cap [-T_N,T_N]} f_u(x) - \E_\P \sum_{x \in X \cap [-T_N,T_N]} uf_1(x)
\end{align}
exist for all $X \in \mathcal{W}$. It remains to prove that the convergence in (\ref{cv:psfu}) and (\ref{cv:psf1}) are uniform in $u$ on compact sets. To this end, we define
\begin{align*}
S(X,u;T_N)&= \sum_{x \in X \cap [-T_N,T_N]} f_u(x)- u f_1(x), \\
H(X;T_N) &= \sum_{x \in X \cap [-T_N, T_N]} f_1(x) - \E_\P \sum_{x \in X \cap[-T_N,T_N]} f_1(x) , \\
I(u;T_N)&= \E_\P  \sum_{x \in X \cap [-T_N,T_N]} ( uf_1(x) - f_u(x) )=\int_{-T_N}^{T_N} (uf_1(x) - f_u(x))K(x,x)d\mu(x),
\end{align*}
and write
\begin{align} \label{eq:sumfudecompo}
\sum_{x \in X \cap [-T_N,T_N]} f_u(x) - \E_\P \sum_{x \in X \cap [-T_N,T_N]} f_u(x) &= S(X,u;T_N) + uH(X;T_N) + I(u;T_N), \\ \label{eq:sumf1decompo}
\sum_{x \in X \cap [-T_N,T_N]} f_u(x) - \E_\P \sum_{x \in X \cap [-T_N,T_N]} uf_1(x) &= S(X,u;T_N) + uH(X;T_N).
\end{align}
From (\ref{cond:h1}) and (\ref{cond:h2}), the sum $S(X,u;T_N)$ is absolutely convergent for all $X \in \mathcal{W}$, uniformly in $u$ on compact sets. The integral $I(u,T_N)$ is also absolutely convergent, uniformly in $u$ on compact sets. Since the left-hand-side of (\ref{eq:sumfudecompo}) (or (\ref{eq:sumf1decompo})) converges as $N \rightarrow + \infty$ for all $X \in \mathcal{W}$, we obtain that $H(X;T_N)$ converges  as $N \rightarrow + \infty$, for all $X \in \mathcal{W}$. Thus, the convergence in (\ref{cv:psfu}) and (\ref{cv:psf1}) are uniform in $u$ on compact sets, and the Lemma is proved.
\end{proof}
\begin{lem} \label{lem:cvSX} For $\P$-almost every $X \in \Conf(E)$, the radius of convergence of the series
\begin{align*}
S_X^\rho(u):= \sum_{k \geq 1} \frac{\rho(\p_k)}{k}u^k
\end{align*}
and
\begin{align*}
S_X^{\tilde{\rho}}(u):= \sum_{k \geq 1} \frac{\tilde{\rho}(\p_k)}{k}u^k
\end{align*}
is at least $R$.
\end{lem}
\begin{proof}
By Fubini's Theorem, we have
\begin{align*}
S_X^\rho(u)=  \sum_{x \in X} f_u(x) - \E_\P \sum_{x \in X} f_u(x)
\end{align*}
and
\begin{align*}
S_X^{\tilde{\rho}}(u)= \sum_{x \in X} f_u(x) - \E_\P \sum_{x \in X} u f_1(x).
\end{align*}
Thus by Lemma \ref{lem:cvunif}, the functions $S_X^\rho$ and $S_X^{\tilde{\rho}}$ are analytic in $D(0,R)$, for $\P$-almost every $X \in \Conf(E)$. The Lemma is proved.
\end{proof}
In order to complete the proof of Proposition \ref{prop:l1}, it remains to prove that 
\begin{align*}
\exp\left( \sum_{i=1}^n S_X^\rho (w_i)-S_X^\rho (z_i) \right), \hspace{0.1cm} \exp\left( \sum_{i=1}^n S_X^{\tilde{\rho}} (w_i)-S_X^{\tilde{\rho}} (z_i) \right) \in L^1(\Conf(E), \P),
\end{align*}
for any $z_1,\dots,z_n,w_1,\dots,w_n \in D(0,R)$. To this aim, we use the following proposition coming from the regularization of multiplicative functionals performed in \cite{bufetovquasisymmetries}.
\begin{prop} \label{prop:regmultfunction} Let $g : E \rightarrow \C \setminus (-\infty,0]$ be a function such that $|g|$ is bounded away from zero and $+\infty$ and verifying 
\begin{align} \label{ineq:normg}
\int_E |g(x)-1|^2 K(x,x)d\mu(x) <+\infty.
\end{align}
Set
\begin{align*}
\overline{S}_{\log g} (X) = \sum_{x \in X} \log g(x) - \E_\P \sum_{x \in X} \log g(x).
\end{align*}
Then we have 
\begin{align*}
\exp \left( \overline{S}_{\log g} (X)\right) \in L^1(\Conf(E), \P).
\end{align*}
\end{prop}
\begin{proof}The existence of $\overline{S}_{\log g} (X)$ follows from condition (\ref{ineq:normg}) and \cite[Proposition 4.1]{bufetovquasisymmetries}. In the sequel, we slightly adapt the argument of \cite{bufetovquasisymmetries},  Proposition 4.6. For $T>0$, we define 
\begin{align*}
g_T(x):= \chi_{[-T,T]}(x)g(x) + \chi_{[-T,T]^c}(x).
\end{align*}
By Fatou's lemma, we have 
\begin{align} \label{ineq:fatou}
\E_\P  \left| \exp  \left( \overline{S}_{\log g} (X) \right)  \right| \leq \liminf_{T \rightarrow + \infty}  \E_\P \left| \exp \left( \overline{S}_{\log g_T} (X) \right) \right|.
\end{align}
Fix $T >0$. Since $g_T-1$ has compact support, both quantities $ \sum_{x \in X} \log g_T(x)$ and $\E_\P \sum_{x \in X} \log g_T(x) $ are well defined. Also observe that 
\begin{multline*}
\left| \exp \left( \overline{S}_{\log g_T} (X) \right) \right| = \exp \left( \Re \sum_{x \in X} \log g_T(x) - \Re \E_\P \sum_{x \in X} \log g_T(x) \right) \\
= \exp \left( \sum_{x \in X} \log |g_T(x)| - \E_\P \sum_{x \in X} \log |g_T(x)| \right) = \exp \left( \overline{S}_{\log |g_T|} (X) \right).
\end{multline*}
Now, we have 
\begin{equation*}
 \E_\P \prod_{x \in X} |g_T(x)| =  \det \left( I + \chi_{[-T,T]} K (|g_T| -1) \right) \leq  \exp  \text{Tr} \left( \chi_{[-T,T]}  K (|g_T|-1) \right),
\end{equation*}
where we used the inequality $1+ u \leq e^u$. We obtain 
\begin{align} \label{ineq:prop1-1}
\log \E_\P \prod_{x \in X} |g_T|(x) \leq \text{Tr} \left( \chi_{[-T,T]} K (|g_T| -1) \right) 
= \int_E (|g_T(x)|-1) K(x,x) d\mu(x).
\end{align}
Since $|g_T|$ is bounded away from zero and $+\infty$, and since $|\log(u)+1-u|/(u-1)^2$ is bounded away from zero and $+\infty$ as long as $u$ is, we have 
\begin{align*}
\left| \int_E (|g_T(x)|-1)K(x,x) d\mu(x) - \int_E \log |g_T(x)| K(x,x) d\mu(x) \right| \leq C \int_E |g_T(x)-1|^2 K(x,x) d\mu(x),
\end{align*}
whence
\begin{multline}\label{ineq:prop1-2}
\E_\P \sum_{x \in X} \log |g_T(x)| = 
\int_E \log |g_T(x)| K(x,x) d\mu(x) \geq \\ \geq \int_E (|g_T(x)|-1)K(x,x)d\mu(x) -C \int_E |g_T(x)-1|^2 K(x,x)d\mu(x).
\end{multline}
Summing the inequalities (\ref{ineq:prop1-1}) and (\ref{ineq:prop1-2}), we obtain 
\begin{multline*}
\log \E_\P \left| \exp \left( \overline{S}_{\log g_T} (X) \right) \right| = \log \E_\P \prod_{x \in X} |g_T(x)| - \E_\P \sum_{x \in X} \log |g_T(x)| \\
\leq C \int_E |g_T(x)-1|^2 K(x,x) d\mu(x) 
\leq C\int_E |g(x)-1|^2 K(x,x) d\mu(x).
\end{multline*}
Recall the inequalities (\ref{ineq:fatou}) and (\ref{ineq:normg}), and the proof is complete.
\end{proof}
We now conclude the proof of Proposition \ref{prop:l1}. Applying the latter Proposition \ref{prop:regmultfunction} to 
\begin{align*}
g(x)= \exp \left( \sum_{i=1}^n f_{w_i}(x)- f_{z_i}(x) \right),
\end{align*}
we obtain that
\begin{align*}
\exp\left( \sum_{i=1}^n S_X^\rho (w_i)-S_X^\rho (z_i) \right) \in L^1(\Conf(E),\P).
\end{align*}
Observing that
\begin{multline*}
\exp\left( \sum_{i=1}^n S_X^{\tilde{\rho}} (w_i)-S_X^{\tilde{\rho}} (z_i) \right)\\
= \exp \left( \sum_{i=1}^{n} \left\{\int_E (f_{w_i}(x)-w_i f_1(x))K(x,x)d\mu(x)  -  \int_E (f_{z_i}(x)-z_if(x)) K(x,x)d\mu(x)\right\}\right)\\
\times \exp\left( \sum_{i=1}^n S_X^\rho (w_i)-S_X^\rho (z_i) \right),
\end{multline*}
we obtain
\begin{align*}
\exp\left( \sum_{i=1}^n S_X^{\tilde{\rho}} (w_i)-S_X^{\tilde{\rho}} (z_i) \right) \in L^1(\Conf(E),\P).
\end{align*}
The proof is complete.
\subsection{A sufficient condition for Giambelli compatibility} \label{sec:giambelli3}
In this section, we prove that in the above-defined context, the Giambelli compatibility of the conditional point processes implies the Giambelli compatibility of the whole point process. Consider a rigid determinantal point process $\P$ on $(E,\mu)$ with a correlation kernel $K$ being the kernel of an orthogonal projection. Let $(f_k)_{k \geq 1}$ be a sequence of functions satisfying conditions (\ref{cond:cvf}), (\ref{cond:cvf1}), (\ref{cond:cvfk}),  (\ref{cond:cvexpf}) and such that there exists a function $h:E \rightarrow [0,+ \infty)$ satisfying (\ref{cond:h1}) and (\ref{cond:h2}). Let $\rho$ and $\tilde{\rho}$ be the $R$-uniform $L^1$-specializations defined by
\begin{align*}
\rho(\p_k)(X)  &=\sum_{x \in X} f_k(x) - \E_\P \sum_{x \in X} f_k(x), \quad X \in \Conf(E), \quad k=1,2,\dots \\
\tilde{\rho}(\p_1)(X) &= \rho(\p_1), \quad \tilde{\rho}(\p_k)(X) = \sum_{x \in X} f_k(x), \quad X \in \Conf(E), \quad k\geq 2,\dots
\end{align*}
For a finite configuration $X \in \Conf(E)$, we set 
\begin{align*}
\rho_0(\p_k)(X)= \sum_{x \in X} f_k(x).
\end{align*}
The main lemma of this section is the following.
\begin{lem} \label{lem:mainlem} Assume that the tail sigma-algebra $\mathcal{F}_{\infty}$ is trivial. If, for $\P$-almost every $X \in \Conf(E)$ and all $T >0$ the pair $(\P( \cdot | X, T), \rho_0)$ is Giambelli compatible, then $(\P, \rho)$ and $(\P,\tilde{\rho})$ are Giambelli compatible.
\end{lem}
We first need the following auxiliary result.
\begin{lem} \label{lem:finite} Let $N \in \N$ and let $\mathbb{P}_N$ be an $N$-point process, i.e. such that
\begin{align*}
\mathbb{P}_N( \#(X) = N) = 1.
\end{align*}
Then for any bounded sequence $a=(a_k)_{k \geq 1} \subset \C$, the specialization $\rho_a$ defined by
\begin{align} \label{eq:defrho0}
\rho_a(\p_k)(X) = \sum_{x \in X} f_k(x)-a_k, \quad k=1,2,\dots
\end{align}
is a $\min(R, 1)$-uniform $L^1$-specialization.
\end{lem}
\begin{proof}
As $\#(X)=N$ for $\mathbb{P}_N$-almost every $X \in \Conf(E)$, we clearly have from condition (\ref{cond:cvfk}) that
\begin{align*}
|\rho_a(\p_k)(X)| \leq N ||f_k||_{L^\infty(E,d\mu)}  + ||a||_{\infty} \leq CNR^{-k} + ||a||_{\infty},
\end{align*}
for $\mathbb{P}_N$-a.e. $X \in \Conf(E)$, where $||a||_\infty = \sup_{k \geq 1} |a_k|$. The series 
\begin{align*}
S_X^{\rho_a}(u)= \sum_{k \geq 1} \frac{\rho_a(\p_k)}{k}u^k
\end{align*}
converges in the disk $D(0,\min(R,1))$ and for $z_1, \dots , z_n, w_1,\dots, w_n \in D(0, \min(R, 1))$, the random variable
\begin{align*}
\exp \left( \sum_{i=1}^n S_X(w_i)- S_X(z_i) \right)
\end{align*}
belongs to $L^\infty(\Conf(E), \mathbb{P}_N)$. The proof is complete.
\end{proof}
We now prove Lemma \ref{lem:mainlem} above.
\begin{proof}[Proof of Lemma \ref{lem:mainlem}]
We only prove that the pair $(\P,\rho)$ is Giambelli compatible, the proof of the Giambelli compatibility of the pair $(\P, \tilde{\rho})$ being similar. Let $T >0$ and let $(T_N)_{N \in \N}$ be an increasing sequence tending to $+ \infty$ as $N \rightarrow + \infty$ such that the limits
\begin{align*}
\lim_{N \rightarrow + \infty} &\sum_{x \in X\cap [-T_N,T_N]} f_k(x) - \E_\P \sum_{x \in X \cap [-T_N,
T_N]} f_k(x), \quad k=1,2,\dots
\end{align*}
exist $\P$-almost surely and in $L^1(\Conf(E),\P)$. Let $N \in \N$ be such that $T_N \geq T$. Set 
\begin{align*}
a_k^{T,N}=  \E_\P \sum_{ x \in X \cap [-T_N,T_N]} f_k(x)- \sum_{x \in X \cap (E \setminus [-T,T]) \cap [-T_N,T_N]} f_k(x), \quad k=1,2,\dots
\end{align*}
The sequence $a^{T,N}=(a_k^{T,N})_{k \geq 1}$ is bounded. Since $\P$ is rigid, the conditioned point process $\P( \cdot | X,T)$ has a fixed number of particles, and by Lemma \ref{lem:finite}, the specialization $\rho_{a^{T,N}}$ is a uniform $L^1$-specialization for $\P(\cdot |X,T)$. By Proposition \ref{prop:shift}, the pair $(\P( \cdot | X,T),\rho_{a^{T,N}})$ is Giambelli compatible and we have 
\begin{multline} \label{eq:lem}
\E_\P \left[ \left. \exp \left( \sum_{i=1}^n S_{X \cap [-T_N,T_N]}^{\rho_{a^{T,N}}}(w_i)- S_{ X \cap [-T_N,T_N]}^{\rho_{a^{T,N}}}(z_i) \right) \right|X,T \right] \\
= \det \left( \frac{1}{z_i-w_j} \right)_{1 \leq i,j \leq n}^{-1} \det \left( \frac{1}{z_i-w_j} \E_\P \left[ \exp \left( S_{X \cap [-T_N,T_N]}^{\rho_{a^{T,N}}}(w_j) - S_{ X \cap [-T_N,T_N]}^{\rho_{a^{T,N}}}(z_i) \right) |X,T \right] \right)_{i,j=1}^n.
\end{multline}
As $N \rightarrow + \infty$, we have from Lemmas \ref{lem:cvf}, \ref{lem:cvunif} and \ref{lem:cvSX} that
\begin{align*}
S_{X \cap [-T_N,T_N]}^{\rho_{a^{T,N}}}(u) \rightarrow S_X^\rho(u)
\end{align*}
$\P$-almost surely and in $L^1(\Conf(E),\P)$. By Proposition \ref{prop:l1}, we can pass to the limit as $N \rightarrow + \infty$ in (\ref{eq:lem}) and we obtain
\begin{multline} \label{eq:lem2}
\E_\P \left[ \left. \exp \left( \sum_{i=1}^n S_X^{\rho}(w_i)- S_X^{\rho}(z_i) \right) \right| X,T \right] \\
= \det \left( \frac{1}{z_i-w_j} \right)_{1 \leq i,j \leq n}^{-1} \det \left( \frac{1}{z_i-w_j} \E_\P \left[ \exp \left( S_X^{\rho}(w_j) - S_X^{\rho}(z_i) \right) |X,T \right] \right)_{i,j=1}^n.
\end{multline}
Now observe that for any $F \in L^1(\Conf(E),\P)$, the sequence
\[(\E_\P [ F |X,T])_{T>0}\]
is a reversed martingale with respect to $(\mathcal{F}_T)_{T >0}$. Since the tail sigma-algebra is trivial, we have
\[\lim_{T \rightarrow + \infty} \E_\P [ F | X,T] = \E_\P [F] \]
$\P$-almost surely and in $L^1(\Conf(E),\P)$. Considering
\[ F(X) = \exp \left( \sum_{i=1}^n S_X^{\rho}(w_i)- S_X^{\rho_a}(z_i) \right) \]
in the left-hand-side of equation (\ref{eq:lem2}), and
\[ F(X) = \exp \left( S_X^{\rho}(w_j) - S_X^{\rho}(z_i) \right), \quad i,j=1,\dots,n, \]
for the right-hand-side, one can pass to the limit $T \rightarrow + \infty$ in (\ref{eq:lem2}) and we obtain equation (\ref{eq:propgiambelli'}). The lemma is proved.
\end{proof}

\section{Proof of Theorem \ref{thm1} and Theorem \ref{thm2}} \label{sec:proofthm}
\subsection{The specializations $\rho^R$ and $\tilde{\rho}^R$, the proof of Proposition \ref{prop:1}, and the equivalence of Theorems \ref{thm1} and \ref{thm2}} \label{sec:proof1}
Let $\P$ be a determinantal point process on $(E,\mu)$ with correlation kernel $K$ being the kernel of an orthogonal projection and satisfying condition (\ref{cond:hilbertschmidt}). Let $R >0$ and consider the sequence of functions $(f_k)_{k \geq 1} \subset L^\infty(E,d\mu)$ defined by 
\begin{align*}
f_k(x) := \frac{1}{(x+ i R)^k}, \quad k=1,2,\dots
\end{align*}
Inequality (\ref{cond:cvf}) is satisfied, and by condition (\ref{cond:hilbertschmidt}), the sequence $(f_k)_{k \geq 1}$ satisfies conditions (\ref{cond:cvf1}) and (\ref{cond:cvfk}). We now compute for $u \in D(0,R)$
\begin{align} \label{form:gf}
f_u(x):= \sum_{k \geq 1} \frac{f_k(x)}{k}u^k = \sum_{k \geq 1} \frac{u^k}{k(x+ i R)^k} = - \log \left(1- \frac{u}{x+ i R} \right).
\end{align}
We thus have
\begin{align*}
f_u(x) = uf_1(x) + O \left( \frac{1}{x^2} \right),
\end{align*}
as $|x| \rightarrow  + \infty$, uniformly in $u$ on compact sets, which by condition (\ref{cond:hilbertschmidt}) implies that conditions (\ref{cond:h1}) and (\ref{cond:h2}) are satisfied with $h(x)= (1+x^2)^{-1}$. Condition (\ref{cond:cvexpf}) follows from condition (\ref{cond:hilbertschmidt}) and formula (\ref{form:gf}). By Proposition \ref{prop:1}, the specializations $\rho^R$ and $\tilde{\rho}^R$ defined by
\begin{align*}
\rho^R(\p_k)(X) = \sum_{x \in X} f_k(x) - \E_\P \sum_{x \in X} f_k(x), \quad X \in \Conf(E),\quad k=1,2 \dots \\
\tilde{\rho}^R(\p_1)= \rho(\p_1), \quad \tilde{\rho}^R(\p_k)(X) = \sum_{x \in X} f_k(x), \quad X \in \Conf(E), \quad k \geq 2,
\end{align*}
are $R$-uniform $L^1$-specializations. In particular, the series
\begin{align*}
S_X^R(u) &:= \sum_{k =1}^{+\infty} \frac{\rho(\p_k(X))}{k}u^k = \sum_{x \in X} f_u(x) - \E_\P \sum_{x \in X} f_u(x) \\ 
\intertext{and} \\
\tilde{S}_X^R(u) &:= \sum_{k =1}^{+\infty} \frac{\tilde{\rho}(\p_k(X))}{k}u^k = \sum_{x \in X} f_u(x) - \E_\P  \sum_{x \in X} u f_1(x)
\end{align*}
converge in $D(0,R)$ for $\P$-almost every $X$ in $\Conf(E)$ and are analytic functions.

From (\ref{form:gf}), we deduce that the function $ u \mapsto -\log g_{u}(x)$ is an analytic continuation of the series $u \mapsto f_u(x)$ on $\C \setminus (\R+iR)$, where $g_{u}$ has been defined in (\ref{eq:defg}). Recalling the definitions of the random functions $\Psi_X(u)$ and $\tilde{\Psi}_X(u)$ in Proposition \ref{prop:1}, we have established the following
\begin{prop} \label{prop:analyticcont} For $\P$-almost every $ X \in \Conf(E)$, the function $\Psi_X(u)$ is an analytic continuation of  
\[\exp \left( -S_X^R(u + i R) \right)
\]
on $\C \setminus \R$, and the function $\tilde{\Psi}_X(u)$ is an analytic continuation of
\[\exp \left( -\tilde{S}_X^{R}(u + i R) \right)
\]
on $\C \setminus \R$
\end{prop}

Proposition \ref{prop:1} is now proved in the same way as Proposition \ref{prop:l1}: applying Proposition \ref{prop:regmultfunction} to the function
\begin{align*}
g(x)= \frac{g_{z_1+i R}(x)\dots g_{z_l+ i R}(x)}{g_{w_1 + i R}(x) \dots g_{w_{l'}+ i R}(x)},
\end{align*}
for any $l,l' \in \N$ and any $z_1,\dots,z_l,w_1,\dots, w_{l'} \in \C \setminus \R$, one obtains points $(i)$ and $(ii)$ of Proposition \ref{prop:1}. The point $(iii)$ is proved in the same way as Lemma \ref{lem:cvunif}.

Proposition \ref{prop:propgiambelli}  and Proposition \ref{prop:analyticcont} now imply the equivalence of Theorems \ref{thm1} and \ref{thm2}.
\subsection{The Giambelli compatibility of orthogonal polynomial ensembles} \label{sec:proof2}
For point processes supported on $N$-points configurations, the specialization $\rho_0^R$ is defined by 
\begin{align*}
\rho_0^R(\p_k)= \sum_{x \in X} f_k(x) =\sum_{x \in X} \frac{1}{(x + i R)^k}.
\end{align*}
Let $\P_N$ be an $N$-orthogonal polynomial ensemble on $E$ with weight $\omega$. We have the following Proposition, analogous to Theorem 3.1 in \cite{giambelli}.
\begin{prop}\label{prop:opegiambelli}The pair $(\P_N,\rho_0^R)$ is Giambelli compatible.
\end{prop}
We will need the following elementary observation.
\begin{prop}\label{prop:schur} For all finite $X = \{x_1, \dots, x_N\} \in \Conf(E)$, and all $\lambda \in \Y$, with $l(\lambda) \leq N$ we have 
\begin{align*}
\rho_0^R(s_\lambda)(X)=\frac{\det \left( (x_k + i R)^{l - \lambda_l -1} \right)_{k,l=1}^N}{\prod_{1 \leq k<l \leq N} (x_k-x_l)}.
\end{align*}
\end{prop}
\begin{proof}From equation (\ref{def:schurpol}), we have 
\begin{align*}
\rho_0^R(s_\lambda)(x_1,\dots,x_N)= \frac{\det\left( (x_k+ i R)^{l-\lambda_l-N}\right)_{k,l=1}^N}{\prod_{1\leq k<l\leq N}((x_k+iR)^{-1}-(x_l+iR)^{-1})}.
\end{align*}
It suffices now to observe that 
\begin{align*}
\prod_{1\leq k<l\leq N}((x_k+iR)^{-1}-(x_l+iR)^{-1})=\frac{\prod_{1\leq k<l\leq N}(x_l-x_k)}{\prod_{k=1}^N (x_k+iR)^{N-1}},
\end{align*}
and the proof is complete.
\end{proof}
\begin{proof}[Proof of Proposition \ref{prop:opegiambelli}]
The proof is adapted from \cite{giambelli}, Theorem 3.1. Using Proposition \ref{prop:schur} and applying the Andr\'eief Formula \cite{andreief}, we have 
\begin{align*}
\E_{\P_N} [\rho_0^R(s_\lambda)(X) ] &= C_N \int_{E^N} \det \left( (x_k + i R)^{l - \lambda_l -1} \right)_{k,l=1}^N \det \left( (x_k+ i R)^{N-l} \right)_{k,l=1}^N d\omega(x_1)\dots d\omega(x_N) \\
&= N! C_N\det \left( \int_E (x+ i R)^{k- \lambda_k-1+N -l} d\omega(x) \right)_{k,l=1}^{N}.
\end{align*}
For $n \in \Z$, set 
\begin{align*}
A^R_n := \int_E (x + iR
)^n d\omega(x).
\end{align*}
We define the semi-infinite matrix $g = (g_{k,l})_{k\geq 0, l=0,\dots N-1}$ by $g_{k,l}=A^R_{N-k-1+l}$, so that 
\begin{align} \label{eq:rhoeps01}
\E_{\P_N} [\rho_0^R(s_\lambda)(X) ] = N! C_N \det (g_{\lambda_{k}-k+N,N-l})_{k,l=1}^{N}.
\end{align}
Observe that we have from Proposition \ref{prop:normalizingconstant} 
\begin{align*}
\det ( g_{k,l})_{k,l=0}^{N-1} = \det ( A^R_{N-k-1+l} )_{k,l=0}^{N-1}= m_N(i R) = m_N = \frac{1}{N!C_N}.
\end{align*}
The square matrix $( g_{k,l})_{k,l=0}^{N-1}$ is thus non-degenerate. We multiply the matrix $g$ to the right by a suitable $N \times N$ matrix so that the resulting matrix, $\tilde{g}$, is lower triangular, and that its diagonal entries equal $1$. Equation (\ref{eq:rhoeps01}) becomes 
\begin{align*}
\E_{\P_N} [\rho_0^R(s_\lambda)(X) ]  =  \det( \tilde{g}_{\lambda_k-k+N,N-l} )_{k,l=1}^N.
\end{align*}
We now set 
\begin{align*}
h_{r,s} = \left\lbrace 
\begin{matrix}
\tilde{g}_{r-s+N-1,N-s-1}, & r\geq 0, \hspace{0.1cm} s=0,\dots, N-1, \\
1, & r=0, \hspace{0.1cm} s \geq N, \\
0, & r<0 , \hspace{0.1cm} s \geq 0.
\end{matrix}
\right.
\end{align*}
By construction, we have 
\begin{align*}
\E_{\P_N} [ \rho_0^R (s_\lambda)(X) ] = \det (h_{\lambda_i-i+j, N-1})_{i,j=1}^N,
\end{align*}
and $h_{0,s}=1$ for all $s \geq 0$ as well as $h_{r,s}=0$ for all $r<0$, $s\geq 0$. By Proposition \ref{prop:giambelligeneral}, we deduce that, if $\lambda =(p_1,\dots,p_d| q_1,\dots, q_d)$, we have 
\begin{align*}
\E_{\P_N} [ \rho_0^\varepsilon(s_\lambda)(X) ] = \det \left( \E_{\P_N}[\rho_0^R(s_{(p_i|q_j)})(X) ] \right)_{i,j=1}^d,
\end{align*}
and the proof is complete.
\end{proof}
\begin{rem}Akemann, Strahov and Wurfel \cite{akemannstrahovwurfel} prove that ‘‘polynomial ensembles’’ are Giambelli compatible. Polynomial ensembles constitute a class of finite determinantal point processes with a non-symmetric correlation kernel, and extend the notion of orthogonal polynomial ensembles.
\end{rem}
\subsection{Conclusion of the proof} \label{sec:proof3}
By Proposition \ref{prop:bufetovconditional}, if $\P$ is a rigid determinantal point process with an integrable kernel, then the conditional measures $\P( \cdot | X,T)$ are orthogonal polynomial ensembles for $\P$-almost every $X \in \Conf(E)$ and all $T >0$. By Proposition \ref{prop:opegiambelli}, $\P(\cdot |X,T)$ is then Giambelli compatible for the specialization $\rho_0^R$ defined by
\begin{align*}
\rho_0^R(\p_k)(X)= \sum_{x \in X} \frac{1}{(x+ i R)^k}, \quad k=1,2,\dots
\end{align*}
From \cite{bufetov-qiu-shamov}, Osada-Osada \cite{osada-osada}, Lyons \cite{lyons}, the tail sigma-algebra of a determinantal point process induced by an orthogonal projection is trivial. By Lemma \ref{lem:mainlem}, the pairs $(\P, \rho^R)$ and $(\P,\tilde{\rho}^R)$ are Giambelli compatible, where 
\begin{align*}
\rho^R(\p_k)(X) &= \rho_0(\p_k)(X) - \E_\P \rho_0(\p_k)(X), \quad k=1,2,\dots\\
\tilde{\rho}^R(\p_1)&= \rho(\p_1), \quad \tilde{\rho}(\p_k)(X) = \rho_0(\p_k)(X), \quad k\geq 2. 
\end{align*}
Theorem \ref{thm2} is proved and Theorem \ref{thm1} follows.

\end{document}